\documentclass[11pt]{article}
\pdfoutput=1

\PassOptionsToPackage{dvipsnames,svgnames}{xcolor}

\usepackage[english]{babel}
\usepackage[utf8]{inputenc}
\usepackage{amsmath,amssymb,amsthm,amsfonts}
\usepackage[colorlinks=true,linkcolor=blue,citecolor=blue,urlcolor=blue]{hyperref}
\usepackage{authblk}

\usepackage[round,authoryear]{natbib}

\usepackage{mathrsfs} 

\usepackage{tikz}
\usetikzlibrary{calc, matrix, positioning}
\usepackage[many]{tcolorbox} 
\usepackage[all]{xy} 
\usepackage{multicol} 
\usepackage{textcomp} 
\usepackage{ytableau}        
\usepackage{placeins} 
\usepackage{enumitem}

\newcommand{\EE}{\mathbb{E}}

\def\multiset#1#2{\ensuremath{\left(\kern-.3em\left(\genfrac{}{}{0pt}{}{#1}{#2}\right)\kern-.3em\right)}}

\newtheorem{thm}{Theorem}[section]
\newtheorem{lemma}[thm]{Lemma}

\theoremstyle{definition}
\newtheorem{defin}[thm]{Definition}

\newtheorem{exa}[thm]{Example}

\newcommand{\MSC}{}
\newcommand{\Keywords}{}
\newcommand{\subjclass}[1]{\gdef\MSC{#1}}
\newcommand{\keywords}[1]{\gdef\Keywords{#1}}

\title{Asymptotics of the Longest Increasing Subsequence in Random Permutations}

\author{Mihir Gupta}
\affil{The Harker School\\
            500 Saratoga Avenue\\
            San Jose, CA 95129\\
            \texttt{mihirgupta292@gmail.com}}

\begin{document}

\maketitle

\begin{abstract}
In this paper, we examine the asymptotic behavior of the longest increasing subsequence (LIS) in a uniformly random permutation of $n$ elements. We rely on the Robinson-Schensted-Knuth correspondence, Young tableaux, and key classical results—including the Erdős-Szekeres theorem and the Hook Length Formula—to demonstrate that the expected LIS length grows as $2\sqrt{n}$. We review the essential variational principles of Logan-Shepp and Vershik-Kerov, which determine the limiting shape of the associated random Young diagrams, and summarize the Baik-Deift-Johansson theorem that links fluctuations of the LIS length to the Tracy-Widom distribution. Our approach focuses on providing conceptual and intuitive explanations of these results, unifying classical proofs into a single narrative and supplying fresh visual examples, while referring the reader to the original literature for detailed proofs and rigorous arguments.
\end{abstract}

\begin{center}
\small \textbf{Mathematics Subject Classification.} \MSC\\[0.25ex]
\end{center}

\section{Introduction}

Consider a permutation of the set $\{1, 2, \ldots, n\}$ viewed as a sequence $(\sigma_1,\sigma_2,\ldots,\sigma_n)$ of distinct integers. Within this sequence, an \emph{increasing subsequence} is defined by indices $1 \le i_1 < i_2 < \cdots < i_k \le n$ such that $\sigma_{i_1}<\sigma_{i_2}<\cdots<\sigma_{i_k}$. Among all increasing subsequences of $\sigma$, the one with maximal length is the \emph{longest increasing subsequence} (LIS), denoted by $L(\sigma)$. Similarly, one may define the \emph{longest decreasing subsequence} (LDS), denoted by $D(\sigma)$.

A classical result of \citet{erdos1935combinatorial} establishes that any sequence of length $n^2+1$ contains an increasing or decreasing subsequence of length at least $n+1$. This result implies that for large $n$, the LIS of a random permutation of $n$ elements is typically of nontrivial length.

Motivated by questions posed by \citet{ulam1961monte}, the problem of determining the expected length of the LIS in a uniformly random permutation has been studied extensively. In the 1970s, \citet{logan1977variational} and \citet{vershik1977asymptotics} independently established that this expected length grows on the order of $2\sqrt{n}$ as $n \to \infty$. \citet{baik1999distribution} later identified the limiting distribution of the suitably normalized LIS, showing its convergence to the Tracy–Widom distribution \citep{tracy1994level} and connecting it to random matrix theory. The emergence of $2\sqrt{n}$ as the precise asymptotic growth rate, rather than another constant times $\sqrt{n}$ (as initially bounded by Hammersley), and the appearance of the Tracy-Widom distribution, typically found in random matrix theory, were particularly remarkable findings, connecting this combinatorial problem to deeper structures in probability and mathematical physics.

In this paper, we provide an overview of these results and outline the central combinatorial and analytical tools that underpin them. We focus on conveying the main ideas and intuition, deferring technical proofs and details to the original references. Our aim is to clarify the connection between elementary combinatorial constructions—such as the Robinson-Schensted-Knuth correspondence and properties of Young tableaux—and the emerging limit laws that govern the asymptotic behavior of the LIS. We achieve this by collecting classical proofs and results into a single, cohesive narrative, augmented with modern examples such as the airplane boarding problem, to offer a fresh perspective for undergraduate readers.

The paper is structured as follows. In Section \ref{sec:RSK}, we recall the Robinson-Schensted-Knuth correspondence. Section \ref{sec:erdos_szekeres} reviews the Erdős-Szekeres Theorem. In Section \ref{sec:hook_length}, we discuss the Hook Length Formula, crucial for counting Young Tableaux. Section \ref{sec:plancherel} delves into the Plancherel measure. Section \ref{sec:ulam_hammersley} covers the Ulam-Hammersley problem and the asymptotic $2\sqrt{n}$ behavior of the LIS length, detailing the contributions of Hammersley, Logan-Shepp, and Vershik-Kerov. Finally, Section \ref{sec:lis_distribution} presents the Baik-Deift-Johansson theorem on the limiting distribution of the LIS.

\subsection{Background}

\subsubsection{Increasing and Decreasing Subsequences}
A \emph{subsequence} of a sequence $A = (a_1, a_2, \ldots, a_n)$ is any sequence 
$A' = (a_{i_1}, a_{i_2}, \ldots, a_{i_k})$ where $1 \leq i_1 < i_2 < \ldots < i_k \leq n$. This subsequence is \emph{increasing} if $a_{i_1} < a_{i_2} < \ldots < a_{i_k}$, and \emph{decreasing} if $a_{i_1} > a_{i_2} > \ldots > a_{i_k}$. Since this paper concerns permutations of distinct integers, all such subsequences are strictly increasing or decreasing. The one with the greatest length is the \emph{Longest Increasing Subsequence} (LIS), and its counterpart is the \emph{Longest Decreasing Subsequence} (LDS).



\subsubsection{Permutations and Randomness}
A \emph{permutation} of $\{1,2,\ldots,n\}$ is a rearrangement of its elements. For example, the six permutations of $\{1,2,3\}$ are:
\[
(1\;2\;3), \; (1\;3\;2), \; (2\;1\;3), \; (2\;3\;1), \; (3\;1\;2), \; (3\;2\;1).
\]
A \emph{random permutation} of $\{1,2,\ldots,n\}$ is chosen uniformly at random from the $n!$ permutations of $\{1,2,\ldots,n\}$.

\subsubsection{Young Diagrams}
A \emph{partition} of a positive integer $n$ is a way of writing $n$ as a sum of positive integers in non-increasing order: $n = \lambda_1 + \lambda_2 + \cdots + \lambda_k$. Such a partition $\lambda = (\lambda_1, \lambda_2, \ldots, \lambda_k)$ can be visualized as a \emph{Young diagram} by arranging boxes in left-justified rows with $\lambda_i$ boxes in row $i$. For instance, the partition $10 = 5 + 3 + 1 + 1$ corresponds to:
\begin{center}
\begin{ytableau}
*(lightgray) & *(lightgray) & *(lightgray) & *(lightgray) & *(lightgray) \\
*(lightgray) & *(lightgray) & *(lightgray) \\
*(lightgray) \\
*(lightgray)
\end{ytableau}
\end{center}

\subsubsection{Young Tableaux}
A \emph{Standard Young Tableau} (SYT) of shape $\lambda$ is formed by placing the numbers $1,2,\ldots,n$ (where $n$ is the total number of boxes in the diagram) into the boxes of the Young diagram so that the entries increase strictly from left to right across each row and from top to bottom down each column.

\begin{exa}
Figure~\ref{fig:syt_example} shows a Standard Young Tableau filled with the numbers $1$ through $9$.
\begin{figure}[ht]
    \centering
    \begin{tikzpicture}
        \definecolor{gridcolor}{rgb}{0, 0, 0}
        
        \foreach \x in {0,1,2} {
            \draw[thick, color=gridcolor] (\x, 0) -- (\x, 1);
        }
        \foreach \x in {0,1,2,3} {
            \draw[thick, color=gridcolor] (\x, 1) -- (\x, 2);
        }
        \foreach \x in {0,1,2,3,4} {
            \draw[thick, color=gridcolor] (\x, 2) -- (\x, 3);
        }
        \foreach \y in {0,1,2,3} {
            \draw[thick, color=gridcolor] (0, \y) -- (2, \y);
        }
        \foreach \y in {1,2,3} {
            \draw[thick, color=gridcolor] (0, \y) -- (3, \y);
        }
        \foreach \y in {2,3} {
            \draw[thick, color=gridcolor] (0, \y) -- (4, \y);
        }
    
        \node at (0.5, 2.5) {1};
        \node at (1.5, 2.5) {3};
        \node at (2.5, 2.5) {5};
        \node at (3.5, 2.5) {8};
    
        \node at (0.5, 1.5) {2};
        \node at (1.5, 1.5) {6};
        \node at (2.5, 1.5) {9};
    
        \node at (0.5, 0.5) {4};
        \node at (1.5, 0.5) {7};
    
        \draw[thick, ->] (-.3, 2.5) -- (-.3, .5);
        \draw[thick, ->] (.5, 3.3) -- (2.5, 3.3);
    \end{tikzpicture}
    \caption{A Standard Young Tableau for the partition $(4,3,2)$ of $n=9$. Entries increase across rows and down columns.}
    \label{fig:syt_example}
\end{figure}
In this tableau, each row and column is strictly increasing. Such combinatorial structures turn out to be intimately connected to permutations and their increasing subsequences.
\end{exa}

\subsubsection{Robinson-Schensted-Knuth (RSK) Correspondence}
The \emph{RSK correspondence} is a bijection between permutations of \(\{1,2,\ldots,n\}\) and pairs of standard Young tableaux (SYTs) of the same shape. This correspondence encodes the increasing and decreasing subsequences of a permutation into the structure of the associated Young tableaux. In particular, for a permutation \(\sigma\), the length of the longest increasing subsequence \(L(\sigma)\) corresponds to the length of the first row of the tableau obtained from \(\sigma\), while the length of the longest decreasing subsequence \(D(\sigma)\) corresponds to the length of its first column.

\begin{exa}
Consider the permutation \(\sigma = (4\;3\;1\;2)\). Applying the RSK correspondence to \(\sigma\) produces two SYTs of the same shape. One of these tableaux is:
\begin{center}
\ytableaushort{12,3,4} 
\end{center}
Here, the first row has length 2, indicating that \(L(\sigma)=2\), and the first column has length 3, indicating \(D(\sigma)=3\). Thus, the structure of one of the tableaux derived from \(\sigma\) directly reveals the lengths of its longest increasing and decreasing subsequences. This tableau and the second tableau will be described in detail in Section \ref{sec:RSK}.
\end{exa}

In the sections that follow, we leverage these combinatorial tools and the insights provided by RSK and Young tableaux to understand the asymptotic behavior of $L(\sigma)$ for random permutations.

\subsubsection{Probabilistic and Analytic Tools}

\paragraph{Poisson process and Poissonisation.}
A \emph{unit-intensity Poisson process} on the plane places points so that
\(\text{Area}(A)\) equals the expected number of points in any Borel set
\(A\), and counts in disjoint sets are independent.
When we study a random permutation of length \(n\) we first let \(n\) be
\(\mathrm{Poisson}(n)\); this step is called \emph{Poissonisation}.
It turns an unwieldy combinatorial model into a geometric model where
independence holds exactly, not approximately.
All statements for fixed \(n\) are obtained at the end by standard
\emph{de-Poissonisation} inequalities.

\paragraph{Asymptotics.}
All limits are as \(n \to \infty\).  
We use four basic tools and nothing more.

\begin{enumerate}[label=(\alph*),leftmargin=*]
  \item \textbf{Big–\(O\).}  
        We write \(a_n = O(b_n)\) when there exists a constant 
        \(C>0\) with \(|a_n| \le C |b_n|\) for all large \(n\).

  \item \textbf{Big–\(\Theta\).}
        We write \(a_n = \Theta(b_n)\) when there exist positive constants \(c_1, c_2\) with \(c_1|b_n| \le |a_n| \le c_2|b_n|\) for all large \(n\). This signifies a tight bound, meaning \(a_n\) grows at the same rate as \(b_n\).

  \item \textbf{Small–\(o\).}  
        We write \(a_n = o(b_n)\) when \(a_n/b_n \to 0\).

  \item \textbf{Stirling's approximation.}  
        \[
           k! \;=\; \sqrt{2\pi k}\,
                 \Bigl(\tfrac{k}{e}\Bigr)^{k}\!
                 \bigl(1 + O(k^{-1})\bigr),
           \qquad k\ge 1.
        \]

  \item \textbf{Markov's inequality.}  
        For any non-negative random variable \(X\) and any \(t>0\),
        \[
           \Pr(X \ge t) \;\le\; \frac{\EE[X]}{t}.
        \]
\end{enumerate}

\paragraph{Variational calculus}
Sometimes we want to know which shape a large random Young diagram is
most likely to resemble.  The question can be rephrased as: choose a
curve \(y=f(x)\) that bounds the diagram after it has been scaled down
to fit inside a unit square, and decide which \(f\) makes a certain
"energy" \(J[f]\) as small as possible.  
A standard rule from calculus says that a curve that minimizes such an
energy must stay put when we give it a tiny wiggle; in symbols the
first-order change of \(J[f]\) is zero.  This condition (called the
Euler–Lagrange equation) boils down to one simple statement for our
problem:

\smallskip
\centerline{\emph{the scaled hook length takes the same value at every
point under the curve}.}

\smallskip
Solving that statement shows the constant must be \(2\).
When the curve is drawn it meets the \(x\)-axis at \(x=2\);
rescaling back to the original size turns this into the familiar
\(2\sqrt{n}\) growth for the average length of the longest increasing
subsequence.

\subsection{Illustrative Applications}\label{sec:illustrative_apps}

\paragraph{Airplane boarding}\label{sec:airplane_boarding}

Beyond purely theoretical contexts, the LIS problem models real-world queueing phenomena, most famously airplane boarding. Let's model a single-aisle plane with seats numbered $1, 2, \ldots, n$ from front to back. The boarding order is a permutation $\sigma = (\sigma_1, \sigma_2, \ldots, \sigma_n)$, where $\sigma_i$ is the seat number of the $i$-th person in line.

An increasing subsequence $\sigma_{i_1} < \sigma_{i_2} < \ldots < \sigma_{i_k}$ with boarding indices $i_1 < i_2 < \ldots < i_k$ represents a potential "traffic jam." The passenger for seat $\sigma_{i_1}$ is ahead in line but has a seat closer to the front than the passenger for seat $\sigma_{i_2}$. This means the aisle is blocked until the first passenger is seated, creating a chain reaction of delays. The LIS length, $L(\sigma)$, measures the size of the worst-case blockage chain.

Boarding strategies correspond to different classes of permutations. "Back-to-front" boarding, $\sigma = (n, n-1, \ldots, 1)$, has an LIS of length 1, minimizing this type of interference. \citet{Bachmat_2006} used this framework to show that more sophisticated strategies—like boarding window seats first, then middle, then aisle—are effective because they create permutations with a demonstrably short LIS, thereby reducing congestion. This makes LIS analysis a valuable tool for optimizing logistics.

\paragraph{Patience sorting}\label{sec:patience_sorting}

The solitaire-style algorithm patience sorting scans a permutation
$\sigma=\sigma_1\ldots\sigma_n$ from left to right, placing each
$\sigma_i$ on the leftmost pile whose top card is larger and
starting a new pile on the right if no such pile exists. This greedy
rule guarantees that, when the last card is dealt, the number of piles
equals the LIS length $L(\sigma)$; storing only the pile tops in a
binary-searched array turns the procedure into the standard
$O(n\log n)$ LIS algorithm. Hence Hammersley's $2\sqrt{n}$ law implies
that a uniformly shuffled $52$-card deck forms about
$2\sqrt{52}\approx14.4$ piles on average—a statistic leveraged in
computer-vision systems that reconstruct occluded playing-card layouts
and in other domains where fast LIS computation is essential.

\subsection{Notation}
We use the following notation and conventions throughout the paper:

\begin{itemize}
    \item $S_n$ denotes the symmetric group on $n$ elements, i.e., all permutations of the set $\{1,2,\ldots,n\}$.
    \item A permutation $\sigma \in S_n$ is written as $\sigma = (\sigma_1,\sigma_2,\ldots,\sigma_n)$, where $\sigma_i$ is the $i$-th element. For such a sequence $\sigma$, we set $|\sigma|=n$.
    \item Unless otherwise stated, a random permutation $\sigma \in S_n$ is chosen uniformly at random.
    \item $L(\sigma)$ is the length of the longest increasing subsequence (LIS) of $\sigma$, and $D(\sigma)$ is the length of the longest decreasing subsequence (LDS) of $\sigma$.
    \item For a random $\sigma \in S_n$, we write $\ell_n = \mathbb{E}[L(\sigma)]$ to denote the expected LIS length.

    \item A partition of \(n\), denoted by \(\lambda = (\lambda_1,\lambda_2,\ldots,\lambda_k)\) with \(\sum_{i=1}^k \lambda_i = n\), is associated with a Young diagram and, when filled with distinct numbers in a strictly increasing manner across rows and columns, with a Standard Young Tableau.
\end{itemize}

\section{The Robinson-Schensted-Knuth (RSK) Correspondence}\label{sec:RSK}

The RSK correspondence is a well-known combinatorial bijection that associates each permutation with a pair of Standard Young Tableaux (SYTs) having the same shape. This correspondence makes it possible to relate the structure of increasing and decreasing subsequences in a permutation to the shape of certain tableaux.

\subsection{Overview}
Consider a permutation $\sigma = (\sigma_1, \sigma_2, \ldots, \sigma_n)$. The RSK correspondence produces a pair $(P,Q)$ of SYTs of the same shape. The tableau $P$ is constructed via a specific \emph{insertion algorithm} that reflects the pattern of increasing subsequences in $\sigma$. The tableau $Q$ records the order in which elements of $\sigma$ are inserted into $P$.

A key property of this correspondence is that the length of the longest increasing subsequence (LIS) of $\sigma$ equals the length of the first row of $P$. Similarly, the length of the longest decreasing subsequence (LDS) equals the length of the first column of $P$.

\subsection{The Insertion Algorithm}
The RSK insertion algorithm proceeds as follows. Start with $P$ and $Q$ empty.

\begin{enumerate}
    \item For each $\sigma_i$ in $\sigma$:
        \begin{enumerate}
            \item Set $r_1 = \sigma_i$.
            \item In $P$, find the first row from top to bottom into which $r_1$ can be inserted:
                \begin{itemize}
                    \item If $r_1$ is larger than every element in that row, place $r_1$ at the end of the row.
                    \item Otherwise, find the first element in the row that is greater than $r_1$. Replace 
                          that element with $r_1$, and let the displaced element become the new $r_2$ to be 
                          inserted into the next row below using the same procedure.
                \end{itemize}
            \item Repeat this "bumping" process down the tableau until an element is appended to the end of some (possibly empty) row.
            \item In $Q$, record the position of each inserted element. The first inserted element is labeled 
                  "1," the second "2," and so forth, so that $Q$ reflects the order of insertions.
        \end{enumerate}
\end{enumerate}

At the end of this process, the pair $(P,Q)$ is the RSK image of $\sigma$. Both $P$ and $Q$ are SYTs of identical shape.
\subsection{Example}
Consider the permutation $\sigma = (2,4,3,7,6,1,5)$, a full permutation of the set $\{1,2,\ldots,7\}$.  
Insert its elements one by one:

1. Insert 2:  
   \[
   P:\,\begin{ytableau}2\end{ytableau}, \qquad 
   Q:\,\begin{ytableau}1\end{ytableau}
   \]

2. Insert 4:  
   \[
   P:\,\begin{ytableau}2 & 4\end{ytableau}, \qquad 
   Q:\,\begin{ytableau}1 & 2\end{ytableau}
   \]

3. Insert 3 (bumps 4 down):  
   \[
   P:\,\begin{ytableau}2 & 3 \\ 4\end{ytableau}, \qquad 
   Q:\,\begin{ytableau}1 & 2 \\ 3\end{ytableau}
   \]

4. Insert 7:  
   \[
   P:\,\begin{ytableau}2 & 3 & 7 \\ 4\end{ytableau}, \qquad 
   Q:\,\begin{ytableau}1 & 2 & 4 \\ 3\end{ytableau}
   \]

5. Insert 6 (bumps 7 down):  
   \[
   P:\,\begin{ytableau}2 & 3 & 6 \\ 4 & 7\end{ytableau}, \qquad 
   Q:\,\begin{ytableau}1 & 2 & 4 \\ 3 & 5\end{ytableau}
   \]

6. Insert 1 (bumps 2, then 4 down):  
   \[
   P:\,\begin{ytableau}1 & 3 & 6 \\ 2 & 7 \\ 4\end{ytableau}, \qquad 
   Q:\,\begin{ytableau}1 & 2 & 4 \\ 3 & 5 \\ 6\end{ytableau}
   \]

7. Insert 5 (bumps 6, then 7 down):  
   \[
   P:\,\begin{ytableau}1 & 3 & 5 \\ 2 & 6 \\ 4 & 7\end{ytableau}, \qquad 
   Q:\,\begin{ytableau}1 & 2 & 4 \\ 3 & 5 \\ 6 & 7\end{ytableau}
   \]

The top row of $P$ has length $3$, so the length of the LIS of $\sigma$ is $3$.  
The first column of $P$ also has length $3$, so the length of the LDS of $\sigma$ is $3$.

\subsection{Bijection and Consequences}
The RSK correspondence is a bijection:%
\[
  \sigma \in S_n \;\longleftrightarrow\; (P,Q),
\]
where $P$ and $Q$ are standard Young tableaux of the same shape.  
Classic proofs of the bijection may be found in Knuth's original
paper \cite{knuth1970permutations} and in Fulton's textbook,
Chapter 4 \cite{Fulton_1996}.  This correspondence relates the
lengths of increasing and decreasing subsequences in $\sigma$ to the
shape of $P$.

\subsection{Key Properties}
In addition to providing a bijection between permutations and pairs of SYTs, the RSK correspondence translates the arrangement of elements in a permutation into the shape and entries of $P$ and $Q$.

\begin{defin}
For a permutation $\sigma$, the \emph{$j$-th basic subsequence} consists of the elements that occupy the $j$-th position in the first row of $P$ during the insertion process. As elements are inserted, each column in the first row collects a sequence of elements. These sequences are the basic subsequences.
\end{defin}

\begin{lemma}[Strict Decrease in Each Basic Subsequence]
\label{BasicDecrease}
Each basic subsequence is strictly decreasing.
\end{lemma}

\begin{proof}
By the insertion rules, when an element is placed into the $j$-th position of the first row, it either replaces a strictly larger element or is appended at the end. Thus, each basic subsequence is formed by elements in strictly decreasing order.
\end{proof}

\begin{lemma}[Linking Adjacent Basic Subsequences]
\label{SchenstedJ-1}
For any element $x$ in the $j$-th basic subsequence with $j \geq 2$, there exists an element $y$ in the $(j-1)$-th basic subsequence such that $y < x$ and $y$ appears before $x$ in $\sigma$.
\end{lemma}

\begin{proof}
When $x$ is placed in the $j$-th position, the element in the $(j-1)$-th position was inserted earlier and is smaller. This ensures that each element in a higher-indexed basic subsequence can be paired with a smaller element from the previous one that appears earlier in the permutation.
\end{proof}

These results imply a structured relationship among the basic subsequences. No increasing subsequence can use more than one element from the same basic subsequence, and each higher-indexed subsequence is connected to a lower-indexed one.

\begin{thm}[LIS and the Number of Columns \cite{schensted1961longest}]
For $\sigma \in S_n$, the length of the longest increasing subsequence $L(\sigma)$ equals the number of columns of $P(\sigma)$.
\end{thm}

\begin{proof}
Let $k$ be the number of columns in $P(\sigma)$. By Lemma \ref{BasicDecrease}, one cannot form an increasing subsequence by taking more than one element from the same basic subsequence. Combined with Lemma \ref{SchenstedJ-1}, it follows that there exists an increasing subsequence of length $k$. Hence, $L(\sigma) = k$.
\end{proof}

\begin{thm}[LDS and the Number of Rows \cite{schensted1961longest}]
For $\sigma \in S_n$, the length of the longest decreasing subsequence $D(\sigma)$ equals the number of rows of $P(\sigma)$.
\end{thm}

\noindent This result follows by considering the RSK image of the reversed permutation $\sigma^r$. 
Applying RSK to $\sigma^r$ produces the transpose of $P(\sigma)$, interchanging rows and columns and relating LDS length in $\sigma$ to the number of rows in $P(\sigma)$.

These relationships show how the RSK correspondence connects the structure of subsequences in a permutation to the shape of the associated Young tableaux.

\section{Erd\H{o}s--Szekeres Theorem}\label{sec:erdos_szekeres}

Erd\H{o}s and Szekeres first proved a result concerning long increasing or decreasing subsequences in permutations. Their theorem \citep{erdos1935combinatorial} is an early contribution to extremal combinatorics. It formed a basis for later studies, including the Ulam–Hammersley problem, which considers the expected length of the longest increasing subsequence in a random permutation. It also influenced later asymptotic results by Logan–Shepp and Vershik–Kerov.

\begin{thm}[\cite{erdos1935combinatorial}]
Let $\sigma \in S_n$ with $n > r \cdot s$ and $r, s \in \mathbb{N}$. Then either the length of the longest increasing subsequence $L(\sigma)$ is greater than $r$, or the length of the longest decreasing subsequence $D(\sigma)$ is greater than $s$.
\end{thm}

\begin{proof}
Suppose, for contradiction, that $\sigma$ is a permutation of $n$ elements with both 
$L(\sigma) \leq r$ and $D(\sigma) \leq s$. By the RSK correspondence, $L(\sigma)$ equals the number of columns and $D(\sigma)$ equals the number of rows in the corresponding Standard Young Tableau (SYT). Thus, the SYT of $\sigma$ fits into an $s \times r$ rectangular shape:

\[
\begin{array}{|c|c|c|c|}
\hline
a_{11} & a_{12} & \cdots & a_{1r} \\
\hline
a_{21} & a_{22} & \cdots & a_{2r} \\
\hline
\vdots & \vdots & \ddots & \vdots \\
\hline
a_{s1} & a_{s2} & \cdots & a_{sr} \\
\hline
\end{array}
\]

Since $n > r \cdot s$, there are at least $r \cdot s + 1$ elements to place. The pigeonhole principle then forces one more element into this diagram, exceeding its $s \times r$ capacity. This would require increasing the number of columns (making $L(\sigma) > r$) or the number of rows (making $D(\sigma) > s$), contradicting the initial assumption.

Therefore, it must be that either $L(\sigma) > r$ or $D(\sigma) > s$.
\end{proof}

\section{The Hook Length Formula}\label{sec:hook_length}

The Hook Length Formula is a celebrated result in combinatorics that provides a simple product formula for the number of Standard Young Tableaux (SYTs) of a given shape. Its importance for the LIS problem stems directly from the RSK correspondence.

Recall that RSK maps the $n!$ permutations in $S_n$ bijectively to pairs of SYTs of the same shape $\lambda \vdash n$. This means the number of permutations that correspond to a particular shape $\lambda$ is precisely $(f^\lambda)^2$, where $f^\lambda$ is the number of SYTs of that shape. Since the LIS length of a permutation is the length of the first row of its associated tableau, the probability of a random permutation having an LIS of length $k$ depends on the sum of $(f^\lambda)^2$ over all shapes $\lambda$ whose first row has length $k$. The Hook Length Formula is the essential tool that unlocks this calculation by giving us an efficient way to compute $f^\lambda$.



A key insight from the RSK correspondence is that the shape of the resulting Young tableaux reveals crucial information about the permutation's structure. Specifically, the length of the first row of the tableau is equal to the length of the Longest Increasing Subsequence (LIS) of the permutation, while the length of the first column corresponds to the length of the Longest Decreasing Subsequence (LDS). Consequently, the Hook Length Formula becomes an essential tool for enumerating permutations with prescribed LIS and LDS lengths, providing a powerful link between algebraic combinatorics and the analysis of permutation patterns.

\subsection{Preliminaries}

Let $f^\lambda$ be the number of SYTs of shape $\lambda \vdash n$, where $\lambda \vdash n$ denotes that $\lambda$ is a partition 
of $n$:
\[
\lambda = (\lambda_1, \lambda_2, \ldots, \lambda_k),
\]
with $\lambda_1 \geq \lambda_2 \geq \cdots \geq \lambda_k > 0$ and $\sum_{i=1}^k \lambda_i = n$.

Under the RSK correspondence, each permutation of $n$ corresponds uniquely to a pair $(P,Q)$ of SYTs of the same shape $\lambda$. Thus, $(f^\lambda)^2$ counts the number of permutations that give rise to shape $\lambda$. Summing over all shapes $\lambda$ that are partitions of $n$, we have:
\begin{equation}
\label{rsk_corresp}
    \sum_{\lambda \vdash n} (f^\lambda)^2 = n!.
\end{equation}

Early work of \citet{macmahon1916combinatory} derived hook–type product
expressions for several special shapes, but a general closed form for
$f^{\lambda}$ was not yet known.  
The full \emph{hook–length formula} was conjectured by J.~S.~Frame in his
1951 thesis and first proved in
\citet{frame1954hook}.

\[
    f^\lambda = \frac{n!}{\prod_{u \in \lambda} h(u)},
\]
where $h(u)$ is the hook length of the cell $u$.

\begin{defin}[Hook and Hook Length]
For a cell $u$ in the Young diagram of $\lambda$, the \emph{hook} of $u$ consists of:
\begin{enumerate}
    \item All cells strictly to the right of $u$ in the same row,
    \item All cells strictly below $u$ in the same column,
    \item The cell $u$ itself.
\end{enumerate}
The \emph{hook length} $h(u)$ of $u$ is the number of cells in this hook.
\end{defin}

\begin{defin}[Walk Types]
These concepts are often used in combinatorial proofs of the Hook Length Formula:
\begin{itemize}
\item A \emph{hook step} moves from one cell to another cell within its hook.
\item A \emph{hook walk} is a sequence of hook steps.
\item A \emph{lattice walk} is a path that moves only to the right or downward to adjacent cells.
\item The \emph{co-hook} of a cell $u$ consists of all cells in the same row to the left of $u$ and all cells in the same column above $u$.
\end{itemize}
These definitions can simplify combinatorial arguments and probabilistic proofs of the Hook Length Formula.
\end{defin}

\subsection*{Example}
Consider $\lambda = (3,2,2)$. The diagram below shows the hook lengths of each cell. The hook of the second cell in the top row (shown in orange) includes the cell itself, one cell to the right, and two cells below, giving it a hook length of 4:

\[
\begin{ytableau} 
5 & *(orange!40)4 & 1 \\ 
3 & 2 \\ 
2 & 1 
\end{ytableau}
\quad \to \quad
\begin{ytableau} 
5 & *(orange!40)4 & *(orange!40)1 \\ 
3 & *(orange!40)2 \\ 
2 & *(orange!40)1 
\end{ytableau}
\]

By determining each cell's hook length, one can compute $f^\lambda$ and count the number of SYTs of shape $\lambda$.

The Hook Length Formula reduces the counting of SYTs to a product of local parameters (the hook lengths):

\begin{thm}[Hook Length Formula \cite{frame1954hook}]
\label{hook_length_formula}
Let $\lambda \vdash n$ and consider its Young diagram. For each cell $u \in \lambda$, let $h(u)$ be the hook length of $u$. Then:
\[
    f^\lambda = \frac{n!}{\prod_{u \in \lambda} h(u)} = \frac{n!}{H(\lambda)},
\]
where $H(\lambda) = \prod_{u \in \lambda} h(u)$.
\end{thm}

\subsubsection{Example}
For $\lambda = (3,2,2)$ with $n=7$, computing the hook lengths and taking their product 
gives $H(\lambda) = 5 \cdot 4 \cdot 1 \cdot 3 \cdot 2 \cdot 2 \cdot 1$. Thus:
\[
f^{(3,2,2)} = \frac{7!}{5 \cdot 4 \cdot 1 \cdot 3 \cdot 2 \cdot 2 \cdot 1} = 21.
\]

For $\lambda = (3,2)$ with $n=5$:
\[
f^{(3,2)} = \frac{5!}{4 \cdot 3 \cdot 2 \cdot 1 \cdot 1} = 5.
\]

The five SYTs of shape $(3,2)$ are:

\begin{center}
\begin{ytableau} 1 & 2 & 3 \\ 4 & 5 \end{ytableau}
\quad
\begin{ytableau} 1 & 3 & 5 \\ 2 & 4 \end{ytableau}
\quad
\begin{ytableau} 1 & 2 & 4 \\ 3 & 5 \end{ytableau}
\quad
\begin{ytableau} 1 & 3 & 4 \\ 2 & 5 \end{ytableau}
\quad
\begin{ytableau} 1 & 2 & 5 \\ 3 & 4 \end{ytableau}
\end{center}

This direct enumeration agrees with the formula's result.
\begin{proof}[Proof Outline of Theorem~\ref{hook_length_formula}]
\citet{greene1982probabilistic} give a probabilistic proof of the Hook Length
Formula based on \emph{hook walks}.  We follow their outline, incorporating the
exposition of \citet{shiyue2019hooklength} (see pp.\,3–4 of the hand‑out) for
the labeling scheme illustrated in Figure~\ref{fig:hook_prob_matrix}.

\smallskip
\textbf{Step 1: Induction set–up.}
In any SYT of size $n$, the entry $n$ must occupy a \emph{corner}
(a cell whose hook length is $1$).  Removing this corner yields an SYT
of shape $\lambda-v$ with $n-1$ boxes, so
\[
  f^\lambda \;=\; \sum_{v\text{ corner of }\lambda} f^{\lambda-v}.
\]
Hence, to prove the Hook Length Formula by induction on $n$, it suffices to
show
\begin{equation}\label{hookinduction-2}
  \frac{n!}{H(\lambda)}
  \;=\;
  \sum_{v\text{ corner of }\lambda}\frac{(n-1)!}{H(\lambda-v)},
\end{equation}
or, equivalently,
\[
  1
  \;=\;\sum_{v\text{ corner of }\lambda}
  \frac{H(\lambda)}{n\,H(\lambda-v)}.
\]

\smallskip
\textbf{Step 2: Hook–walk probabilities.}
Interpret the summands in~\eqref{hookinduction-2} as probabilities.
Fix a corner $v$.  For any starting cell $u$, let $P(u,v)$ be the
probability that a \emph{hook walk} beginning at $u$ terminates at $v$:
at every intermediate cell $w$ choose uniformly among the $h(w)-1$
other cells in its hook.  A typical hook walk is shown in
Figure~\ref{fig:hook_walk_path}.

\begin{figure}[ht]
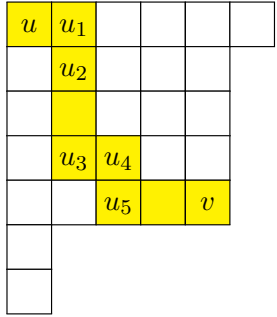

    \centering
    \begin{ytableau}
    *(yellow) u & *(yellow) u_1 & *(white) & *(white) & *(white) & *(white) \\
    *(white) & *(yellow) u_2 & *(white) & *(white) & *(white) \\
    *(white) & *(yellow) & *(white) & *(white) & *(white) \\
    *(white) & *(yellow) u_3 & *(yellow) u_4 & *(white) & *(white) \\
    *(white) & *(white) & *(yellow) u_5 & *(yellow) & *(yellow) v\\
    *(white) \\
    *(white)
    \end{ytableau}
    \caption{A hook walk from cell $u$ to the corner cell $v$.}
    \label{fig:hook_walk_path}
\end{figure}

\smallskip
\textbf{Step 3: The rectangle labeling.}
All walks that end at $v$ stay inside the rectangle with north‑west
corner $u$ and south‑east corner $v$.  Whenever four cells form the
corners of a $2\times2$ rectangle (Figure~\ref{fig:hook_length_rect}),
their hook lengths satisfy
$h(a)+h(d)=h(b)+h(c)$, hence
$(h(a)-1)+(h(d)-1)=(h(b)-1)+(h(c)-1)$.

\begin{figure}[ht]
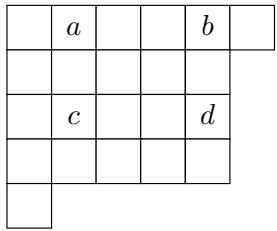

    \centering
    \begin{ytableau}
    *(white) & *(white) a & *(white) & *(white) & *(white) b & *(white) \\
    *(white) & *(white) & *(white) & *(white) & *(white) \\
    *(white) & *(white) c & *(white) & *(white) & *(white) d\\
    *(white) & *(white) & *(white) & *(white) & *(white) \\
    *(white)
    \end{ytableau}
    \caption{A $2\times2$ rectangle with $h(a)+h(d)=h(b)+h(c)$.}
    \label{fig:hook_length_rect}
\end{figure}

For any cell $w$ in a Young diagram, we define its $\operatorname{arm}(w)=a(w)$ as the number of cells to its right in the same row, and its $\operatorname{leg}(w)=b(w)$ as the number of cells below it in the same column. The hook length is then $h(w)=a(w)+b(w)+1$.

\smallskip
\textbf{Boundary notation.}
Fix a corner cell $v$ and look at its \emph{co-hook}, i.e.\ the cells in
the same row \emph{to the left} of $v$ and the cells in the same column
\emph{above} $v$ (excluding $v$ itself).

\noindent Suppose there are

\textbullet\; $k$ cells to the left of $v$; list them left–to–right as
$w_1,\dots,w_k$ and set
\[
  x_i \;=\; h(w_i)-1 \;=\; a(w_i)+b(w_i),
  \qquad 1\le i\le k .
\]

\textbullet\; $\ell$ cells above $v$; list them top–to–bottom as
$w'_1,\dots,w'_\ell$ and set
\[
  y_j \;=\; h(w'_j)-1 \;=\; a(w'_j)+b(w'_j),
  \qquad 1\le j\le \ell .
\]

Thus each boundary cell is coded by the quantity $x_i$ or $y_j$, equal to
its hook length minus 1.





\smallskip
\textbf{Label assignment.}
Label every cell $w$ by the reciprocal of the number of \emph{available
moves} from that cell:
\[
   \frac{1}{h(w)-1}
   \;=\;
   \frac{1}{a(w)+b(w)} .
\]
On the boundary this means the labels are exactly
$\tfrac{1}{x_i}$ and $\tfrac{1}{y_j}$; interior labels are then forced by
the rectangle rule.  Figure~\ref{fig:hook_prob_matrix} shows the general
pattern and a concrete numerical example.


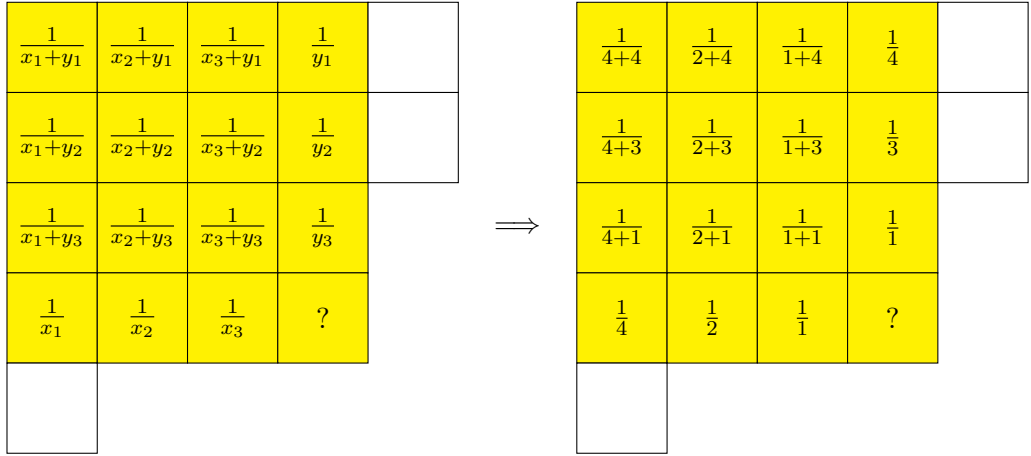
\begin{figure}[ht]
\centering
\begin{tikzpicture}
\matrix [matrix of math nodes,
        nodes={minimum size=1.2cm, outer sep=0pt, fill=yellow, draw, anchor=center},
        row sep=-\pgflinewidth, column sep=-\pgflinewidth] (m1) {
\tfrac{1}{x_1+y_1} & \tfrac{1}{x_2+y_1} & \tfrac{1}{x_3+y_1} & \tfrac{1}{y_1} & |[fill=white]| \\
\tfrac{1}{x_1+y_2} & \tfrac{1}{x_2+y_2} & \tfrac{1}{x_3+y_2} & \tfrac{1}{y_2} & |[fill=white]| \\
\tfrac{1}{x_1+y_3} & \tfrac{1}{x_2+y_3} & \tfrac{1}{x_3+y_3} & \tfrac{1}{y_3} \\
\tfrac{1}{x_1} & \tfrac{1}{x_2} & \tfrac{1}{x_3} & ? \\
|[fill=white]| \\
};
\node[right=0.2cm of m1] (arrow) {$\Longrightarrow$};
\matrix [matrix of math nodes,
        nodes={minimum size=1.2cm, outer sep=0pt, fill=yellow, draw, anchor=center},
        row sep=-\pgflinewidth, column sep=-\pgflinewidth, right=0.2cm of arrow] (m2) {
\tfrac{1}{4+4} & \tfrac{1}{2+4} & \tfrac{1}{1+4} & \tfrac{1}{4} & |[fill=white]| \\
\tfrac{1}{4+3} & \tfrac{1}{2+3} & \tfrac{1}{1+3} & \tfrac{1}{3} & |[fill=white]| \\
\tfrac{1}{4+1} & \tfrac{1}{2+1} & \tfrac{1}{1+1} & \tfrac{1}{1} \\
\tfrac{1}{4} & \tfrac{1}{2} & \tfrac{1}{1} & ? \\
|[fill=white]| \\
};
\end{tikzpicture}
\caption{The cell labeling scheme for the probabilistic hook-walk proof labels each cell $w$ by the reciprocal of its available moves, $1/(h(w)-1)$. The left panel shows the general algebraic pattern determined by the co-hook of a corner cell, while the right provides a concrete numerical example with boundary values $x=(4,2,1)$ and $y=(4,3,1)$ derived from a specific Young diagram.}
\label{fig:hook_prob_matrix}
\end{figure}

\smallskip
\textbf{Step 4: Lattice paths.}
Restrict temporarily to lattice paths that move only one square south
or east.  If $w(P)$ denotes the product of labels along a path~$P$ from
$u$ to $v$, then (by induction on $k+\ell$)
\[
  \sum_{P:u\to v} w(P)
  \;=\;
  \frac{1}{x_1x_2\cdots x_k\,y_1y_2\cdots y_\ell}.
\]
For $k=\ell=1$ this is illustrated in
Figure~\ref{fig:lattice_walk_2x2}.

\begin{figure}[ht]
\centering
\begin{tikzpicture}
\matrix [matrix of math nodes,
         nodes={minimum size=1.2cm, outer sep=0pt, fill=white, draw, anchor=center},
         row sep=-\pgflinewidth, column sep=-\pgflinewidth] {
\scriptstyle \tfrac{1}{x_1+y_1} & \scriptstyle \tfrac{1}{y_1} & |[fill=white]| \\
\scriptstyle \tfrac{1}{x_1} & \scriptstyle ? \\
};
\end{tikzpicture}
\caption{Sum over the two lattice paths equals
$\frac{1}{x_1y_1}$.}
\label{fig:lattice_walk_2x2}
\end{figure}

\smallskip
\textbf{Step 5: From lattice paths to hook walks.}
A hook walk may skip rows or columns.  Expanding
\[
  \prod_{i=1}^{k}\Bigl(1+\tfrac{1}{x_i}\Bigr)
  \prod_{j=1}^{\ell}\Bigl(1+\tfrac{1}{y_j}\Bigr)
\]
selects which rows/columns are skipped; every monomial is the weight of
a hook walk that starts at the north‑west corner of the corresponding
sub‑rectangle and ends at $v$.  Hence
\[
  \sum_{u} P(u,v)
  \;=\;
  \prod_{t\in\operatorname{cohook}(v)}
     \Bigl(1+\tfrac{1}{h(t)-1}\Bigr)
  \;=\;
  \frac{H(\lambda)}{H(\lambda-v)}.
\]

\smallskip
\textbf{Step 6: Summing over all corners.}
Since every hook walk terminates at some corner,
$\sum_{v}P(u,v)=1$ for each $u$.  Summing first over $u$ and then over
$v$,
\[
  n \;=\;
  \sum_{u\in\lambda}\sum_{v}P(u,v)
      \;=\;
  \sum_{v}\frac{H(\lambda)}{H(\lambda-v)}.
\]
Multiplying by $(n-1)!$ and rearranging yields
\[
  \frac{n!}{H(\lambda)}
    \;=\;
  \sum_{v\text{ corner of }\lambda}
     \frac{(n-1)!}{H(\lambda-v)},
\]
which is exactly the inductive identity~\eqref{hookinduction-2}.
Therefore the Hook Length Formula holds.
\end{proof}

By interpreting hook length ratios as probabilities of certain hook walks, we obtain a probabilistic and combinatorial argument for the Hook Length Formula.

This explicit formula for $f^\lambda$ allows us to analyze the probability distribution on shapes induced by the Plancherel measure, which we consider in the next section.

\section{Plancherel Measure}\label{sec:plancherel}

Plancherel measure provides a probability distribution on the set of all partitions of an integer $n$. Consider a uniformly random permutation $\sigma_n \in S_n$. Applying the Robinson--Schensted--Knuth (RSK) correspondence to $\sigma_n$ produces a pair of standard Young tableaux of the same shape. Denote by $\lambda^{(n)}$ the shape of the standard Young tableau obtained from $\sigma_n$ under RSK. Thus, $\lambda^{(n)}$ is the shape of a partition of $n$ that corresponds to a randomly chosen permutation of $\{1, 2, \ldots, n\}$.

For any fixed partition $\lambda \vdash n$, let $f^\lambda$ denote the number of standard Young tableaux of shape $\lambda$. The Plancherel measure is defined as
\[
    P(\lambda^{(n)} = \lambda) = \frac{(f^\lambda)^2}{n!}.
\]

This measure characterizes the distribution of $\lambda^{(n)}$. In particular, the length of the longest increasing subsequence $L(\sigma_n)$ of a random permutation $\sigma_n \in S_n$ has the same distribution as the length of the first row $\lambda_1^{(n)}$ of the partition $\lambda^{(n)}$.

A more detailed explanation of this equivalence, as well as its implications for the asymptotic behavior of $L(\sigma_n)$, is discussed in the following section.




\section{Ulam–Hammersley Problem}\label{sec:ulam_hammersley}

The Ulam–Hammersley problem concerns the asymptotic behavior of the expected length of the Longest Increasing Subsequence (LIS) in a uniformly random permutation. For each $n$, define:
\begin{equation}
    \ell_n = \mathbb{E}[L(\sigma_n)] = \frac{1}{n!} \sum_{\sigma \in S_n} L(\sigma),
\end{equation}
where $\sigma_n$ is chosen uniformly at random from $S_n$.

To gain an initial sense of how $\ell_n$ grows, we note its exact values for small $n$:
\[
\begin{array}{ccccc}
    \ell_1 = 1.00, & \ell_2 = 1.50, & \ell_3 = 2.00, & \ell_4 = 2.41, & \ell_5 = 2.79 \\
    \ell_6 = 3.14, & \ell_7 = 3.47, & \ell_8 = 3.77, & \ell_9 = 4.06, & \ell_{10} = 4.33.
\end{array}
\]
These values already suggest that $\ell_n$ grows at a sublinear rate—it increases more slowly than any linear function of $n$. However, it still appears to grow without bound. Ulam first posed the question of determining the asymptotic behavior of $\ell_n$ in 1961, and Hammersley undertook a systematic study in 1970. As a result, the problem is known as the Ulam--Hammersley problem.

\subsection{First Bounds by Hammersley}

\citet{hammersley1972few} established the first nontrivial asymptotic bounds for $\ell_n$. 
He showed that:
\begin{equation}\label{eq:hammersley_bound}
    \frac{\pi}{2} \leq \liminf_{n\to\infty} \frac{\ell_n}{\sqrt{n}} \leq \limsup_{n\to\infty} \frac{\ell_n}{\sqrt{n}} \leq e.
\end{equation}

\subsubsection{Lower Bound}%
\label{subsec:hammersley-lower}

To obtain Hammersley's constant $\pi/2$ we "Poissonize'' the permutation
problem: place a unit–intensity Poisson point process in a square of
\emph{area}~$n$ (side length $\sqrt n$).\par
\textbf{Poisson primer.} A unit–intensity Poisson process has two key
properties:
\begin{enumerate}[label=(\roman*)]
  \item The expected number of points in any region equals its area.
  \item Counts in disjoint regions are independent.
\end{enumerate}
\textit{Example.} Imagine flicking darts uniformly at the square so that, on
average, one dart lands per unit area; the exact dart counts in two
non-overlapping rectangles are independent.  

Now we trace a monotone path that always jumps to the nearest point strictly to
the north–east. Let $Q_0$ be the south–west corner and define recursively $Q_{i+1}=Q(Q_i)$,
where $Q(P)$ is the nearest Poisson point in the cone
$\{(x,y):x>P_x,\;y>P_y\}$.

\paragraph{Distribution of one step.}
Write the increment $Q_iQ_{i+1}$ in polar coordinates
$(R,\varphi)$ with $R>0$ and angle $0<\varphi<\pi/2$.  
Because the Poisson process is isotropic, $R$ and $\varphi$ are independent.  
We derive their laws step by step:

\begin{enumerate}[label=(\alph*)]
  \item \textbf{Survival probability.}  
        The north–east quarter–disk of radius $r$ has area
        $\dfrac{\pi r^{2}}{4}$.  The probability that it contains \emph{no}
        Poisson points is therefore
        \[
          \Pr(R>r)=\exp\!\Bigl(-\tfrac{\pi r^{2}}{4}\Bigr).
        \]

  \item \textbf{Radial density.}  
        Differentiating the survival function gives
        \[
          f_R(r)
          =-\frac{d}{dr}\Pr(R>r)
          =\frac{\pi r}{2}\,e^{-\pi r^{2}/4},
          \qquad r>0.
        \]
        A one-line check shows $\int_{0}^{\infty}f_R(r)\,dr=1$.

  \item \textbf{Mean radius.}
        \[
          \mathbb{E}[R]
          =\int_{0}^{\infty} r\,f_R(r)\,dr
          =\frac{\pi}{2}\int_{0}^{\infty} r^{2}e^{-\pi r^{2}/4}\,dr
          =1.
        \]

  \item \textbf{Angular density and mean.}  
        Rotational symmetry inside the quadrant makes
        \(
          f_\varphi(\theta)=\dfrac{2}{\pi},\;
          0<\theta<\pi/2,
        \)
        whence
        \[
          \mathbb{E}[\cos\varphi]
          =\frac{2}{\pi}\int_{0}^{\pi/2}\cos\theta\,d\theta
          =\frac{2}{\pi}.
        \]
\end{enumerate}

Combining (c) and (d) and using independence,
\[
  \mathbb{E}[R\cos\varphi]
  =\mathbb{E}[R]\,\mathbb{E}[\cos\varphi]
  =1\cdot\frac{2}{\pi}
  =\frac{2}{\pi}.
\]
By symmetry the vertical projection has the same expectation.

\paragraph{From one step to many.}
After $m$ steps the horizontal (and vertical) displacement is the sum of
$m$ independent and identically distributed increments.  By the law of large numbers this is
$m\cdot2/\pi$ in the limit $m\to\infty$.
The walk stops once either coordinate reaches $\sqrt{n}$, so
\[
  m\;\approx\;\frac{\pi}{2}\sqrt{n} \quad (n\to\infty).
\]

\paragraph{Consequences for permutations.}
Each greedy jump adds one element to an increasing subsequence, so the LIS
length satisfies $L_n\ge m$ for \emph{all} realizations.  Hence
\[
  \mathbb{E}[L_n]\;\ge\;\frac{\pi}{2}\sqrt{n},
  \qquad\text{and therefore}\qquad
  \liminf_{n\to\infty}\frac{\mathbb{E}[L_n]}{\sqrt{n}}\;\ge\;\frac{\pi}{2}.
\]
(Variational methods of Logan–Shepp and Vershik–Kerov later raise the constant
to~$2$, but $\pi/2$ remains a valid uniform lower bound.)

\subsubsection{Upper Bound}

We now establish the upper bound of $e$ on the growth rate of $\ell_n$. Our argument is inspired by the approach in \citet{romik2015surprising}.

Consider $X_{n,k}$, the number of increasing subsequences of length $k$ in a random 
permutation $\sigma_n \in S_n$. Since $X_{n,k}$ counts the subsets of size $k$ in increasing order, observe that:
\[
\mathbb{E}[X_{n,k}] = \frac{\binom{n}{k}}{k!}.
\]
The binomial coefficient $\binom{n}{k}$ counts the $k$-subsets of $\{1,2,\ldots,n\}$, and each $k$-subset has a $1/k!$ probability of appearing in increasing order out of the $k!$ possible permutations of that subset.

To bound $P(L(\sigma_n) \geq k)$, we use Markov's inequality:
\[
P(L(\sigma_n) \geq k) = P(X_{n,k} \geq 1) \leq \mathbb{E}[X_{n,k}] = \frac{\binom{n}{k}}{k!}.
\]
To approximate for large $n$ and $k$, we use Stirling's approximation:
\[
k! \sim \sqrt{2\pi k}\left(\frac{k}{e}\right)^k.
\]
As $n$ and $k$ become large, a rough estimate gives:
\[
\frac{\binom{n}{k}}{k!} \approx \frac{n^k}{(k!)^2}.
\]
To show that $L(\sigma_n)$ typically does not exceed $e\sqrt{n}$ by more than a constant factor, set
\[
k = \lfloor (1+\delta)e\sqrt{n} \rfloor
\]
for some small $\delta > 0$. As $n \to \infty$, we have $k \sim (1+\delta)e\sqrt{n}$. Thus:
\[
P(L(\sigma_n) \geq k) \leq \frac{n^k}{(k!)^2}.
\]
Applying Stirling's approximation to $k!$ and using the fact that $k \sim (1+\delta)e\sqrt{n}$:
\[
k! \sim \sqrt{2\pi k}\left(\frac{k}{e}\right)^k \sim \sqrt{2\pi (1+\delta)e\sqrt{n}}\,( (1+\delta)\sqrt{n} )^{(1+\delta)e\sqrt{n}}.
\]
Substituting back:
\[
\frac{n^k}{(k!)^2} \approx \frac{n^{(1+\delta)e\sqrt{n}}}{\left(\sqrt{2\pi (1+\delta)e\sqrt{n}}\,( (1+\delta)\sqrt{n} )^{(1+\delta)e\sqrt{n}}\right)^2}.
\]
After cancelling the common factor $n^{k}$, the denominator still contributes $(1+\delta)^{2k}\!\sqrt{2\pi k}$; this extra exponential factor forces the ratio to decay to zero as $n \to \infty$.

Hence:
\[
P(L(\sigma_n) \geq (1+\delta)e\sqrt{n}) \to 0 \quad \text{as } n \to \infty.
\]
This shows that it is highly unlikely for $L(\sigma_n)$ to be larger than $(1+\delta)e\sqrt{n}$ for 
any fixed $\delta > 0$. Consequently:
\[
\limsup_{n \to \infty} \frac{\ell_n}{\sqrt{n}} \leq e.
\]
Combined with Hammersley's lower bound, this shows:
\[
\frac{\pi}{2} \;\leq\; \liminf_{n\to\infty} \frac{\ell_n}{\sqrt{n}} \;\leq\; \limsup_{n\to\infty} \frac{\ell_n}{\sqrt{n}} \;\leq\; e,
\]
providing both lower and upper constraints on the growth of the expected LIS length.

Hammersley conjectured that the following limit exists:
\[
c = \lim_{n \to \infty} \frac{\mathbb{E}[L(\sigma_n)]}{\sqrt{n}}.
\]
He further conjectured that $c = 2$. This conjecture was subsequently proven correct by 
\citet{logan1977variational} and \citet{vershik1977asymptotics} independently.

\subsection{Confirming Hammersley's Conjecture}

\citet{logan1977variational} and \citet{vershik1977asymptotics} independently proved Hammersley's conjecture, establishing that:
\[
    \lim_{n \to \infty} \frac{\mathbb{E}[L(\sigma_n)]}{\sqrt{n}} = 2.
\]

We follow the exposition in \citet{stanley2005increasing} to outline the key ideas.

\subsubsection{Expected Length of the LIS in Terms of Partitions}

Recall that under the Robinson--Schensted--Knuth (RSK) correspondence, every permutation
$\sigma\in S_n$ is mapped to a pair of standard Young tableaux (SYT) that
share a common shape $\lambda\vdash n$.  Denote by $f^\lambda$ the number
of SYTs of shape $\lambda$.  Because this map is a bijection onto pairs of
tableaux, we have the classical identity
\[
  n!\;=\;\sum_{\lambda\vdash n} (f^\lambda)^2.
\]

The longest increasing subsequence (LIS) length of $\sigma$ equals the first
row length $\lambda_1$ of the associated shape.  Consequently
\[
  \mathbb{E}[L(\sigma_n)]
  \;=\;\frac{1}{n!}\sum_{\lambda\vdash n}\lambda_1\,(f^\lambda)^2.
\]

\paragraph{A dominant partition.}
Let $p(n)$ be the number of partitions of $n$ and set
$M_n:=\max_{\lambda\vdash n}f^\lambda$.  Because the sum of squares of all
$f^\lambda$ equals $n!$, Cauchy--Schwarz yields
\[
  M_n^2\;\le\;\sum_{\lambda\vdash n}(f^\lambda)^2=n!
  \quad\Longrightarrow\quad
  M_n\le\sqrt{n!}.
\]
Conversely, at least one summand reaches the average value:
\[
  M_n^2
  \;\ge\;
  \frac{1}{p(n)}\sum_{\lambda\vdash n}(f^\lambda)^2
  =\frac{n!}{p(n)}
  \quad\Longrightarrow\quad
  M_n\ge\sqrt{\frac{n!}{p(n)}}.
\]

By the Hardy--Ramanujan asymptotic
$p(n)=\exp\!\bigl(O(\sqrt n)\bigr)$\,\cite{HardyRamanujan1918}, the gap
between these bounds is only a sub-exponential factor:
\[
  \sqrt{n!}\,e^{-O(\sqrt n)}
  \;\le\;M_n\;\le\;\sqrt{n!}.
\]

Hence $M_n=\Theta(\sqrt{n!})$; that is, a single dominant partition
$\lambda^{(n)}$ contributes a summand of the same order as the entire sum.
This observation underlies many refinements of the RSK-based analysis of the
expected LIS length.

\paragraph{Key Observation: The Variational Problem.}
With $\ell_n = \mathbb{E}[L(\sigma_n)]$, the character identity gives 
\[ 
\ell_n = \frac{1}{n!}\sum_{\lambda\vdash n}\lambda_1(f^\lambda)^2 \approx \frac{\lambda^{(n)}_1(f^{\lambda^{(n)}})^2}{n!}, 
\]
where $\lambda^{(n)}$ is the partition with maximal $f^\lambda$. We see that the dominant partition controls the average LIS, and since $(f^{\lambda^{(n)}})^2$ is only a sub-exponential factor below $n!$, it follows that $\ell_n \approx \lambda^{(n)}_1$. Thus the growth of $\ell_n$ reduces to understanding the first-row length of $\lambda^{(n)}$. By the Hook-Length Formula, 
\[ 
f^\lambda = \frac{n!}{\prod_{u\in\lambda} h(u)}, 
\]
where $h(u)$ is the hook length of a cell $u$. Maximizing $f^\lambda$ is therefore equivalent to minimizing the product of hook lengths, which transforms the question into a variational problem to find the limit shape of the Young diagram of $\lambda^{(n)}$.



\textbf{Normalization and the Limit Shape.}  
To analyze the limit as \(n \to \infty\), the Young diagram is scaled so that its total area is 1. Assigning each cell a side length of \(1/\sqrt{n}\) results in a continuous shape fitting into a unit square. Under this scaling, the boundary of the Young diagram \(\lambda^{(n)}\) converges to a limit shape described by a curve \(y = \Psi(x)\).

If the limit shape intersects the $x$-axis at \(x=b\), it follows that:
\[
    c := \lim_{n \to \infty} \frac{\mathbb{E}[L(\sigma_n)]}{\sqrt{n}} \geq b.
\]
However, as noted by Logan–Shepp and Vershik–Kerov, it is necessary to solve a variational problem to determine the exact limiting shape. Their results show that the limiting curve intersects the $x$-axis at \(x=2\). This yields:
\[
    \lim_{n \to \infty} \frac{\mathbb{E}[L(\sigma_n)]}{\sqrt{n}} \geq 2,
\]
consistent with Hammersley's conjecture.

\textbf{Logan–Shepp and Vershik–Kerov's Variational Result:}  
\citet{logan1977variational} and \citet{vershik1977asymptotics}, working independently, identified the limiting curve \(y=\Psi(x)\) by formulating and solving a variational problem.

Consider a function \(f\) that parametrizes the limiting shape of the Young diagram. The normalized hook-length at a point \((x,y)\) in the scaled diagram can be expressed as
\[
   f(x) - y + f^{-1}(y) - x,
\]
reflecting the geometric constraints imposed by hook lengths.

They defined the functional
\[
    I(f) \;=\; \iint_A \log\!\bigl(f(x) - y + f^{-1}(y) - x\bigr)\,dx\,dy,
\]
where \(A\) is the region under the scaled diagram, subject to the
area-normalization
\[
    \iint_A dx\,dy \;=\; 1.
\]
Minimizing \(I(f)\) over admissible \(f\) characterizes the limiting shape
\(\Psi(x)\).

\paragraph{How the minimization is solved}
Here is a six-step road map from the "minimize an integral" statement to the
explicit limit curve; a full derivation appears in
\cite[Section\,2]{logan1977variational} and the textbook account
\cite[1.13–1.15]{romik2015surprising}.

\begin{enumerate}[label=\arabic*.]
\item \textbf{Rotate and rescale.}
      Divide both coordinates by $\sqrt n$ and rotate $45^{\circ}$ so the
      boundary becomes a non-increasing, $1$-Lipschitz curve
      $f:[0,1]\!\to\![0,1]$ enclosing \emph{area one}.

\item \textbf{Define the functional.}
      For $(x,y)$ inside the diagram set
      $h_f(x,y)=f(x)-y+f^{-1}(y)-x$ (its \emph{hook length} after scaling).
      The quantity to minimize is
      \[
        I(f)=\iint_{A}\log h_f(x,y)\,dx\,dy,
      \]
      where $A=\{(x,y):0\le y\le f(x)\}$.

\item \textbf{Use a Lagrange multiplier.}
      Fix the area by adding $\alpha(\text{area}-1)$:
      \[
        J(f)=I(f)+\alpha\!\Bigl(\iint_A dx\,dy-1\Bigr).
      \]

\item \textbf{First variation $\Rightarrow$ Euler–Lagrange equation.}
      Perturb $f$ to $f+\varepsilon\eta$ with a smooth bump~$\eta$ and demand
      $\frac{d}{d\varepsilon}J(f+\varepsilon\eta)\bigl|_{\varepsilon=0}=0$
      for \emph{all} $\eta$.  The resulting Euler–Lagrange condition is
      \[
        h_f(x,y)=f(x)-y+f^{-1}(y)-x=C,
        \qquad(x,y)\in A,
      \]
      meaning every hook length in the optimal diagram is the same constant~$C$.

\item \textbf{Find $C$.}
      Integrating $h_f\equiv C$ over the unit-area region forces $C=2$.

\item \textbf{Solve for $f$.}
      On the boundary put $y=f(x)$ and eliminate $f^{-1}$, obtaining a
      first-order ordinary differential equation whose solution is the parametric curve
      \[
        x = y + 2\cos\theta,\qquad
        y = \frac{2}{\pi}\bigl(\sin\theta-\theta\cos\theta\bigr),
        \qquad 0\le \theta \le \pi.
      \]
      Vershik–Kerov derive the same curve via an entropy maximization
      argument \cite[Sec.\,3]{vershik1977asymptotics}.
\end{enumerate}

Figure~\ref{fig:parametric_curve} illustrates this limiting boundary and the
normalized hook length at point \((x,y)\).  Its intersection with the
\(x\)-axis at \((2,0)\) confirms that
\[
    \lim_{n\to\infty} \frac{\mathbb{E}[L(\sigma_n)]}{\sqrt{n}} \;\ge\; 2,
\]
consistent with Hammersley's conjecture.

\begin{figure}[ht]
\centering
\includegraphics[width=0.75\textwidth]{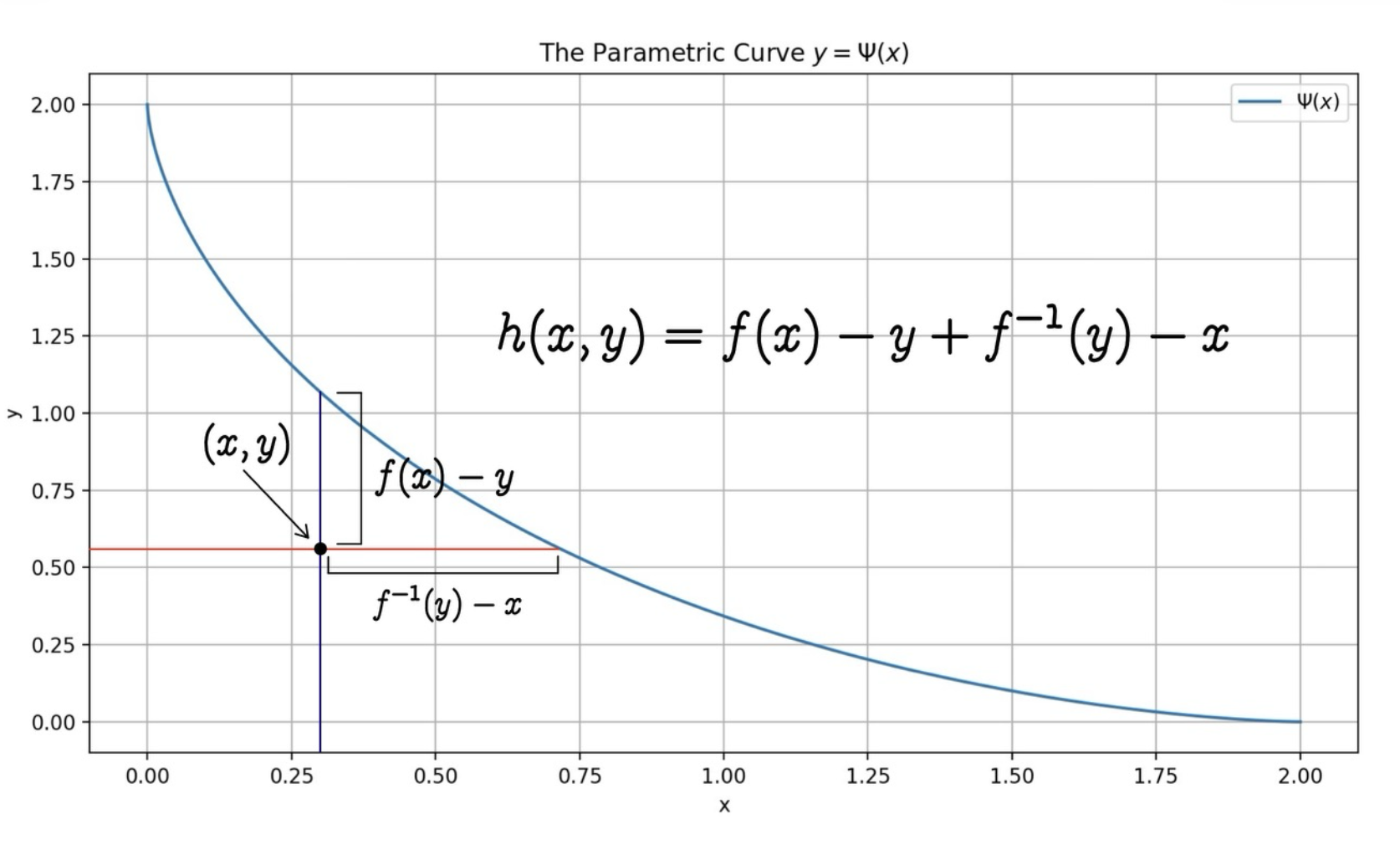}
\caption{The limiting boundary of the Normalized Young Diagram and the Normalized Hook Length. This curve, derived by \citet{logan1977variational} and \citet{vershik1977asymptotics}, illustrates the asymptotic shape whose first row length (scaled) corresponds to the constant 2 in the $2\sqrt{n}$ growth of the LIS.}
\label{fig:parametric_curve}
\end{figure}
This curve intersects the $x$-axis at $x = 2$, suggesting that $c \geq 2$.

The asymptotic form shows \(\mathbb{E}[L(\sigma_n)] \sim 2\sqrt{n}\).

\begin{lemma}[{\cite{logan1977variational}}]
For a uniform random permutation $\sigma_n\in S_n$,
\[
     \liminf_{n\to\infty}\frac{L(\sigma_n)}{\sqrt{n}}\;\ge\;2.
\]
\end{lemma}

\begin{proof}
The Logan--Shepp limit–shape theorem states that the rescaled Young diagram
lies, with probability~$1$, inside any $\varepsilon$–neighbourhood of the
parametric curve displayed above.  Because that curve meets the $x$-axis at
$(2,0)$, the first–row length must eventually satisfy
$\lambda_1\ge (2-\varepsilon)\sqrt n$.  Letting $\varepsilon\downarrow0$
gives the claimed $\liminf$ inequality.
\end{proof}

\vspace{0.5em}

\begin{lemma}[{\cite{vershik1977asymptotics}}]
For the same random permutations,
\[
     \limsup_{n\to\infty}\frac{L(\sigma_n)}{\sqrt{n}}\;\le\;2.
\]
\end{lemma}

\begin{proof}
Vershik and Kerov showed that the Plancherel‐distributed diagram cannot rise
\emph{above} the limit curve by more than~$\varepsilon$ with probability
tending to~$1$.  Since any point with $x>2+\varepsilon$ lies outside that
tube, the first row must eventually obey
$\lambda_1\le (2+\varepsilon)\sqrt n$.  Sending $\varepsilon\downarrow0$
establishes the $\limsup$ bound.
\end{proof}

Combining these two lemmas shows that the limit exists and equals 2, thus proving Hammersley's conjecture.

\begin{thm}[\cite{logan1977variational}, \cite{vershik1977asymptotics}]
For a uniformly random permutation \(\sigma_n\), the expected length of the longest increasing subsequence satisfies:
 
\[
    \lim_{n \to \infty} \frac{\mathbb{E}[L(\sigma_n)]}{\sqrt{n}} = 2.
\]
This confirms Hammersley's conjecture.
\end{thm}

\section{The Limiting Distribution of the LIS}\label{sec:lis_distribution}

The determination that $\mathbb{E}[L_n] \sim 2\sqrt{n}$ is a landmark result, yet it characterizes only the first moment of the LIS length distribution. A complete probabilistic description requires understanding the nature of the fluctuations of $L_n$ around this asymptotic mean. This refined inquiry shifts the focus from the expected value to the limiting law of the centered and scaled random variable. The pivotal work of Baik, Deift, and Johansson resolved this question, demonstrating that the fluctuations occur on a scale of $n^{1/6}$ and converge to the Tracy-Widom distribution, a law first discovered in the realm of random matrix theory.

\subsection{The Baik–Deift–Johansson Theorem}

Baik, Deift, and Johansson proved that, after centering and normalizing by $n^{1/6}$, the fluctuations of the LIS length converge to the \emph{Tracy–Widom distribution}, a probability law that also governs the largest eigenvalue of certain random matrices.

\begin{thm}[{\cite{baik1999distribution}}, 1999]
For a uniformly random permutation \(\sigma_n\in S_n\) and every real \(t\),
\[
    \lim_{n\to\infty}
    \mathbb{P}\!\left(
        \frac{L(\sigma_n)-2\sqrt{n}}{n^{1/6}}\le t
    \right)
    \;=\;
    F_{\mathrm{TW}}(t),
\]
where \(F_{\mathrm{TW}}\) denotes the Tracy--Widom distribution function.
\end{thm}

\paragraph{Painlevé\,II representation of \(F_{\mathrm{TW}}\).}
The cumulative distribution function is given by
\[
 F_{\mathrm{TW}}(t)= \exp\!\Biggl(-\int_t^{\infty}(x-t)\,u(x)^2\,dx\Biggr),
\]
where \(u(x)\) is the solution to the Painlevé II equation
\[
 u''(x)=2u(x)^3+xu(x),
\]
with the asymptotic behavior \(u(x) \sim -\frac{e^{-2x^{3/2}/3}}{2\sqrt{\pi}\,x^{1/4}}\) as \(x \to \infty\) \cite{forrester2015painleve}.


\subsubsection*{Consequences for mean and variance}

Let
\[
  \mu_{\mathrm{TW}} = \int t\,dF_{\mathrm{TW}}(t)
  \quad\text{and}\quad
  \sigma_{\mathrm{TW}}^{2}
      =\int t^{2}\,dF_{\mathrm{TW}}(t)
       -\Bigl(\int t\,dF_{\mathrm{TW}}(t)\Bigr)^{2}.
\]
Numerically,
\(
  \mu_{\mathrm{TW}}\approx -1.77108680
\)
and
\(
  \sigma_{\mathrm{TW}}^{2}\approx0.81319479
\).
These constants yield the refined asymptotics
\[
   \mathbb{E}\!\bigl[L(\sigma_n)\bigr]
        = 2\sqrt{n} \;+\; \mu_{\mathrm{TW}}\,n^{1/6} + o\!\bigl(n^{1/6}\bigr),
   \qquad
   \operatorname{Var}\!\bigl[L(\sigma_n)\bigr]
        = \bigl(\sigma_{\mathrm{TW}}^{2}+o(1)\bigr)\,n^{1/3}.
\]

Accordingly, the deviations of \(L(\sigma_n)\) from its deterministic
\(\,2\sqrt{n}\) limit occur on the \(n^{1/6}\) scale and are governed, in law,
by the universal Tracy--Widom distribution.


\setlength{\bibhang}{2em}
\setlength{\parindent}{0pt}
\setlength{\parskip}{0.5ex}
\bibliographystyle{abbrvnat}
\bibliography{refs}

\end{document}